\documentclass[12pt]{amsart}
\usepackage[utf8]{inputenc}
\usepackage[all]{xy}
\usepackage{graphicx}
\usepackage{amsmath,amsthm,bm,mathrsfs,amscd,tikz}
\usepackage{amsfonts,amssymb,mathrsfs}
\usepackage{verbatim}
\usepackage{float}
\usepackage{enumerate}
\RequirePackage{color}
\usepackage{xcolor}
\usepackage{soul}
\usepackage{comment}
\usepackage[colorlinks=true, pdfstartview=FitV, linkcolor=purple, citecolor=blue,
filecolor=violet, urlcolor=violet]{hyperref}
\numberwithin{figure}{section}
\numberwithin{table}{section}

\theoremstyle{plain}
\newtheorem{theorem}{Theorem}[section]
\newtheorem{prop}[theorem]{Proposition}

\theoremstyle{definition}
\newtheorem{remark}[theorem]{Remark}

\newtheorem{definition}[theorem]{Definition}
\newtheorem{example}[theorem]{Example}

\def \A{{\mathbb A}}
\def \R{{\mathbb R}}
\def \Q{{Q'}}
\def \D1{{\Delta_{\ge -1} }} 
 
\def \p{{\mathtt p}}
\def \cc{{\mathbf c}}
\def \gg{{\mathbf g}}
\def \dd{{\mathbf d}}

\newcommand{\C}{\mathcal{C}}

\DeclareMathOperator{\m}{mod}

\DeclareMathOperator{\rep}{rep}

\begin{document}

\title[Polytopal realizations of non-crystallographic associahedra]{Polytopal realizations of \\ non-crystallographic associahedra}
\author[A.~Felikson]{Anna Felikson}
\address{Department of Mathematical Sciences, Durham University, Upper Mountjoy Campus, Stockton Road, Durham, DH1 3LE, UK}
\email{anna.felikson@durham.ac.uk} 
\thanks{Research was supported in part by the Leverhulme Trust research grant RPG-2019-153 (PT) and a Royal Society Wolfson Research Merit Award (EY)}

\author[P.~Tumarkin]{Pavel Tumarkin}
\address{Department of Mathematical Sciences, Durham University,  Upper Mountjoy Campus, Stockton Road, Durham, DH1 3LE, UK}
\email{pavel.tumarkin@durham.ac.uk}

\author[E.~Yıldırım]{Emine Yıldırım}
\address{School of Mathematics, University of Leeds, Leeds, LS2 9JT, UK}
\email{pmtey@leeds.ac.uk}


\begin{abstract}
We use the folding technique to show that generalized associahedra for non-simply-laced root systems (including non-crystallographic ones) can be obtained as sections of simply-laced generalized associahedra constructed by Bazier-Matte, Chapelier-Laget, Douville, Mousavand, Thomas and Yıldırım. 
  
  \end{abstract}

\maketitle

\section{Introduction}

An associahedron is a polytope whose combinatorics is defined by families of Catalan objects.
For example, vertices of an associahedron can be indexed by triangulations of a convex polygon, edges correspond to flips of triangulations, and faces correspond to dissections (i.e., partial triangulations). Combinatorial description goes back to Tamari~\cite{T} and Stasheff~\cite{St}, since then many different polytopal realizations were constructed (see the paper by Ceballos, Santos and Ziegler~\cite{CSZ}, the recent survey by Pilaud, Santos and Ziegler~\cite{PSZ} and extensive bibliography therein).  

Generalized associahedra were introduced combinatorially by  Fomin and Zelevinsky \cite{FZY} for any finite (crystallographic) root system in the context of cluster algebras;  when restricted to type $A$ root system, the construction gives the classical associahedron. The combinatorial construction was extended to all Coxeter elements by Marsh, Reineke and Zelevinsky~\cite{MRZ} (the initial construction in~\cite{FZY} corresponds to a bipartite Coxeter element), and was reformulated without the crystallographic assumption by Fomin and Reading~\cite{FR}, see also~\cite{R2}.

The first polytopal realizations of generalized associahedra (for a bipartite Coxeter element) were constructed by Chapoton, Fomin and Zelevinsky~\cite{CFZ}, the normal fans of these polytopes are $\dd$-vector fans~\cite{FZ2}. Holweg, Lange and Thomas~\cite{HLT} provided realizations of generalized associahedra for arbitrary Coxeter element, the normal fans of these polytopes are Cambrian fans~\cite{R,RS}, which coincide with $\gg$-vector fans~\cite{FZ4} for finite root systems. More general constructions were considered by Holweg, Pilaud and Stella in~\cite{HPS}.  

In~\cite{BCDMTY}, Bazier-Matte, Chapelier-Laget, Douville, Mousavand, Thomas and Yıldırım provided a construction of generalized associahedra for simply-laced root systems (and any Coxeter element) by using representation theory of quivers. More precisely, they consider a vector space with coordinates indexed by almost positive roots, and use certain \emph{mesh relations} (which depend on a number of parameters) in the corresponding cluster category  to produce an affine subspace of linear dependencies. Intersection of this affine subspace with the positive orthant results in a realization of a generalized associahedron, where the normal fan is precisely the $\gg$-vector fan  (we recall the construction in more details in Section~\ref{prelim}). The construction above extends the earlier construction of the classical associahedron by Arkani-Hamed, Bai, He and Yan appearing in~\cite{ABHY}. The space of all polytopal realizations of $\gg$-vector fans was explored by Padrol, Palu, Pilaud and Plamondon in~\cite{P4}.    

\bigskip

In this note, we show that generalized associahedra for non-simply-laced root systems (including non-crystallographic ones) can be obtained as sections of simply-laced generalized associahedra presented in~\cite{BCDMTY}. A non-simply-laced root system $\Delta'$ can be obtained from a certain simply-laced root system $\Delta$ by a folding technique~\cite{L}, see~\cite{Stem} for details. The construction in~\cite{BCDMTY} depends on a number of parameters which can be indexed by positive roots in a simply-laced root system $\Delta$; we consider ``symmetric'' collections of parameters giving rise to a ``symmetric'' generalized associahedron $\A_\Delta$. We then construct a certain plane $\Pi$ using the restrictions coming from the folding technique, and consider the section of $\A_\Delta$ by $\Pi$. The main result can be formulated as follows (see Section~\ref{thm} for the details).

\begin{theorem}[Theorem~\ref{main}]
The section of $\A_\Delta$ by $\Pi$ is a realization of the generalized associahedron $\A_{\Delta'}$. The normal fan of  $\A_{\Delta'}$ is the corresponding $\gg$-vector fan, and it can be obtained as an orhogonal projection of the normal fan of $\A_\Delta$ onto $\Pi$.

  \end{theorem}

  We note that in the case $\Delta'$ is crystallographic, the theorem can be derived from the results of Arkani-Hamed, He and Lam~\cite{AHL}. 
  
  
We recall all necessary terms and constructions in Section~\ref{prelim}, formulate our main result in Section~\ref{thm} and give the proofs in the final Section~\ref{proof}.

\subsection*{Acknowledgements}
We are grateful to Nathan Reading and Salvatore Stella for many helpful discussions. We also thank Nathan Reading for very useful comments on the first version of the paper. The work was initiated at the Isaac Newton Institute for Mathematical Sciences, Cambridge; we are grateful to the organizers of the program “Cluster algebras and representation theory”, and to the Institute for support and hospitality during the program; this work was supported by EPSRC grant no EP/R014604/1.

\section{Preliminaries}
\label{prelim}

We start with describing the construction of the generalized associahedra for the simply-laced Dynkin diagrams by~\cite{BCDMTY}. We then remind the essential facts about $\gg$-vectors and weighted unfoldings. While talking about generalized associahedra we will omit the word ``generalized'' if there is no ambiguity.

\subsection{Simply-laced generalized associahedra}

In~\cite{BCDMTY}, authors use representation theoretical techniques to generate certain equations and their intersection of the positive orthant to obtain associahedra. We recall  the key parts of the construction below, see~\cite{BCDMTY} for more details.

Let $Q$ be an orientation of a simply-laced Dynkin diagram, and denote by $\rep Q$ the category of representations of $Q$. Let $D^b(\rep Q)$ be the bounded derived category of $\rep Q$. Note that $\rep Q$ can be embedded in $D^b(\m kQ)$ by sending each indecomposable object $M$ in $\rep Q$ to a complex in  $D^b(\rep Q)$ whose only nonzero object is $M$. We consider a subcollection $\C_Q$ of $D^b(\m kQ)$, where we take the all indecomposable objects in $\rep Q$ along with the shifted projectives. This subcollection $\C_Q$ can be thought as \emph{cluster category} without some morphisms but we do not need to go into details in this paper. Let $N$ be the number of indecomposables in $\C_Q$, and let $V$ be a real vector space of dimension $N$ (note that $N=n(h+2)/2$, where $h$ is the corresponding Coxeter number and $n$ is the rank of $Q$). Consider an indeterminate $t_M$ for each object $M\in \C_Q$ as a coordinate function in $V$, and positive constants $c_{M}$ for each representation $M \in \rep Q$. An associahedron is then constructed for each choice of parameters $c_M$. Given a collection of parameters $c_M$, we consider the following hyperplanes in $\R^n$: 

\[t_{\tau M}+\ t_{M}=\sum_{E} t_{E} + c_M,\] where $E\in\C_Q$ sits in an Auslander-Reiten sequence $0\rightarrow \tau M \rightarrow E \rightarrow M \rightarrow 0$ in $D^b(\rep Q)$ for the objects in $\C_Q$. 

It has been shown in~\cite{BCDMTY} that the intersection of this set of hyperplanes with the positive orthant provides a construction of a generalized associahedron. For type $A$, one recovers the classical associhedra as constructed in~\cite{ABHY}.

\begin{example}
  \label{cat-a3}
    Let $Q=\xymatrix{1 \ar[r] & 2 & 3\ar[l]}$. Then the category $C_Q$ can be drawn as follows.

\[
\xymatrix{ t_{11}\ar[dr] &   {\color{blue} c_{11}}     & t_{21}\ar[dr] & {\color{blue} c_{21}} & t_{31\ar[dr]} &\\
           & t_{12}\ar[dr]\ar[ur] & {\color{blue} c_{12}} & t_{22}\ar[dr]\ar[ur] & {\color{blue} c_{22}} & t_{32}\\
           t_{13}\ar[ur] &   {\color{blue} c_{13}}     & t_{23}\ar[ur] & {\color{blue} c_{23}} & t_{33}\ar[ur] &
}
\]

The equations are
\begin{align*}
    t_{11}+\ t_{21} &= t_{12} + c_{11}\\
    t_{13}+\ t_{23} &= t_{12} + c_{13}\\
    t_{12}+\ t_{22} &= t_{21} + t_{23} + c_{12}\\
    t_{21}+\ t_{31} &= t_{22} + c_{21}\\
    t_{23}+\ t_{33} &= t_{22} + c_{23}\\
    t_{22}+\ t_{32} &= t_{31} + t_{33} + c_{22}
\end{align*}

The intersection of the affine subspace generated by these equations and the positive orthant gives us the associahedron of type $A_3$.

We show an explicit solution when all the parameters $c_{kj}$ are set to be $1$ in Figure~\ref{fig:A3}. To be able to draw the associahedron, we consider the projection to the last three coordinates, i.e. $t_{31},\ t_{32}$ and $t_{33}$.

\begin{figure}[h]
\includegraphics[width=9cm]{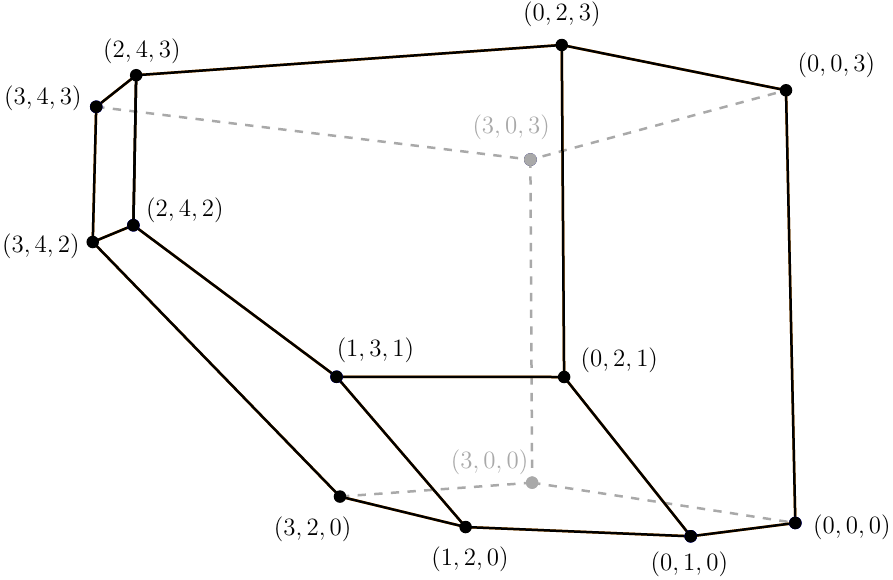}
\caption{Associahedron of type $A_3$, see Example~\ref{cat-a3}.}
\label{fig:A3}
\end{figure}

\end{example}

The equations above can be rephrased using combinatorics of root systems as follows. Let $Q$ be an orientation of a simply-laced Dynkin quiver and $\Delta$ be a root system corresponding to $Q$ with $\Delta^+$ being the set of positive roots. Let $\Delta_{\geq -1}$ denote the set of almost positive roots (i.e., the set of positive roots together with negative simple roots). Then objects in $\C_Q$ are in a bijective correspondence with elements of $\Delta_{\geq -1}$, so indeterminates $t_M$ for each $M\in \C_Q$ can be understood as $t_{\alpha}$ for each $\alpha\in \Delta_{\geq -1}$, and the constants $c_M$ for each $M \in \rep Q$ can be understood as $c_{\alpha}$ for each $\alpha \in \Delta^+$. Using the compatibility degree  $(\cdot || \cdot)$ on $\Delta_{\geq -1}$ (see e.g.~\cite{YZ} or~\cite{S}), the equations can then be rewritten as

\begin{equation}~\label{x-eqs}
    t_{\beta} + t_{\alpha}=\sum_{\gamma} t_{\gamma} + c_{\alpha},
\end{equation}
where  $(\beta || \alpha) = 1$ and $\gamma\in \Delta_{\geq -1}$ such that $(\gamma || \alpha) = (\gamma || \beta) =0$. 
In other words, in the language of cluster combinatorics the indeterminates $t_{\alpha}$ are indexed by cluster variables $x_\alpha$ of the corresponding cluster algebra. 

The same equations can also be obtained by tropicalization of $u$-equations~\cite{AHL}.

\subsection{$\gg$-vectors}
\label{g}
Given a cluster algebra with principal coefficients, every cluster variable can be assigned with an integer {\em $\gg$-vector} as a multi-degree of the Laurent expansion with respect to the initial seed~\cite{FZ4}. A {\em $\gg$-vector fan} is a simplicial fan spanned by $\gg$-vectors with maximal cones corresponding to seeds of the cluster algebra. The $\gg$-vector fan is complete if and only if the cluster algebra is of finite type~\cite{FZ4}. A generalized associahedron in the crystallographic case can be realized as a polytope whose normal fan is the $\gg$-vector fan~\cite{R,RS,HLT,HPS}. In particular, the normal fan to the realization of the generalized associahedron in the construction of~\cite{BCDMTY} is precisely the corresponding $\gg$-vector fan.

In the skew-symmetric case, $\gg$-vectors satisfy the equations~(\ref{x-eqs}) when all the parameters $c_{kj}$ are set to zero (\cite{BCDMTY}). 

In the non-crystallographic cases $\gg$-vectors and  $\gg$-vector fans were defined in~\cite{DT1,DT2} via Nakanishi--Zelevinsky tropical duality~\cite{NZ}. It was shown that the definitions are compatible with quiver mutations, and also with weighted unfoldings (see Section~\ref{unf}).


\subsection{Unfoldings}
\label{unf}

Our main tool is a {\em weighted unfolding}, see~\cite{DT1} for the general definition (which is a generalization of the notion of unfolding introduced by Zelevinsky, see~\cite{FST2}). For the purposes of this paper, we are interested in finite types only. Every non-simply-laced root system $\Delta'$ (including non-crystallographic ones) can be obtained from certain (non-unique) simply-laced root system $\Delta$ by a folding technique (see~\cite{L,Cr,Stem} for details). Under folding (which is a linear map), a simply-laced root system $\Delta$ is taken to a union $\bigcup w_k\Delta'$, where $\{w_k\}$ is a certain set of weights (in the crystallographic case all $w_j=1$, so the map can be considered as a projection; see~\cite{MP,DT1,DT2} for weights in the non-crystallographic cases). For convenience, we reproduce the list of unfoldings with their weights in Table~\ref{weights}.

\begin{table}
  \caption{\scriptsize {\normalsize Unfoldings of non-simply-laced root systems and their weights.}\vphantom{$\int\limits_.$}\\
            For crystallographic non-simply-laced root systems we use weighted quiver notation: the corresponding exchange matrices are non skew-symmetric, and $\alpha\stackrel{a,b}{\longrightarrow} \beta$ means that $d_\alpha a=d_\beta b$ in the symmetrizing diagonal matrix $(d_i)$. For non-crystallographic root systems, the exchange matrix is skew-symmetric, and the weights of arrows are equal to the moduli of matrix elements. Quivers for unfoldings of dihedral types are oriented from white to black. $U_i$ denote the Chebyshev polynomials of the second kind, and $\varphi=2\cos(\pi/5)$ is the golden ratio. Quiver $Q'$ is defined up to sink-source mutations, the corresponding composite mutations should be applied to $Q$.}

  \label{weights}
 \includegraphics[width=\textwidth]{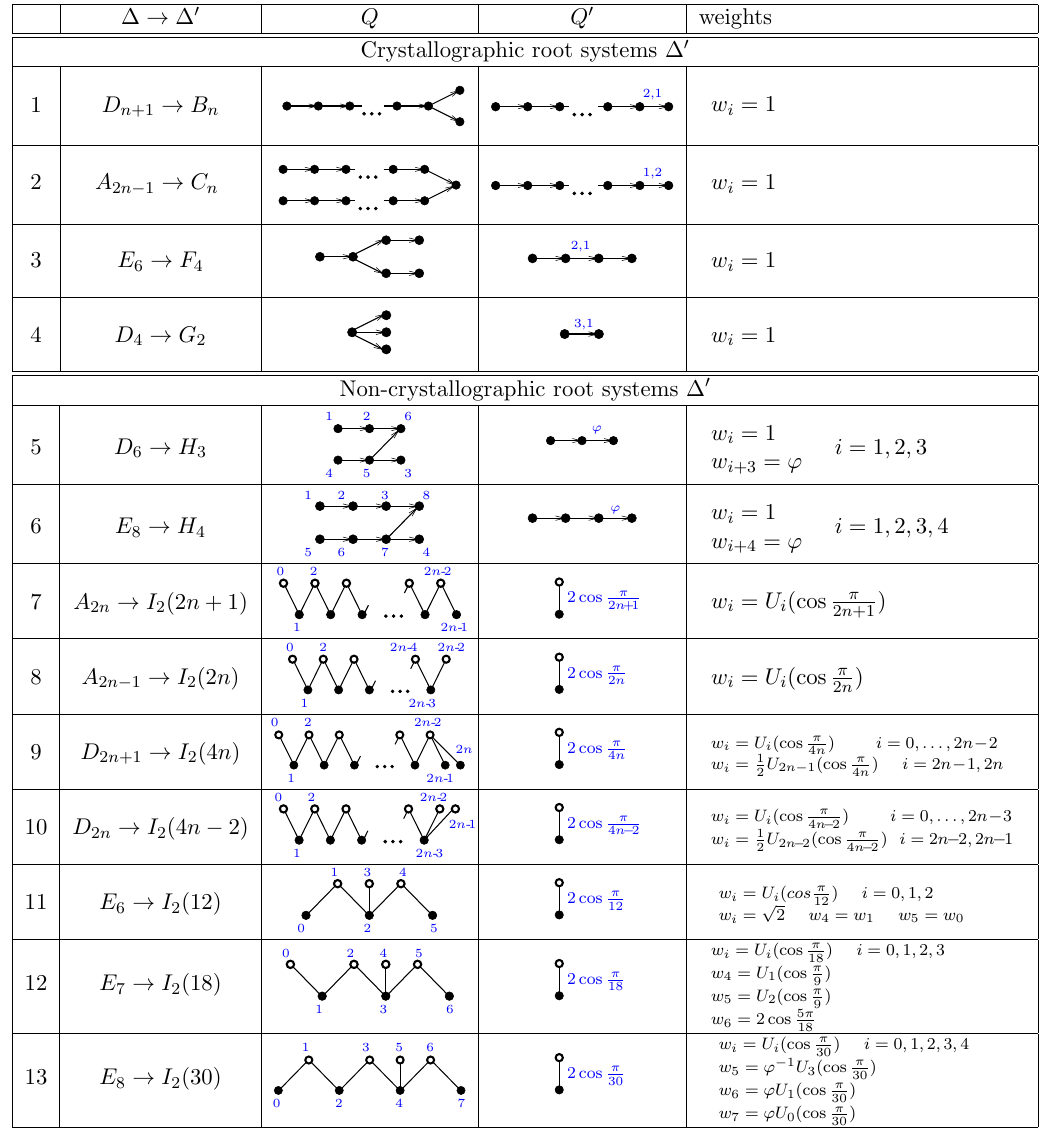}
\end{table}

Weighted unfoldings are compatible with mutations of quivers (or diagrams) and seeds \cite{FST2, Du, DT1}.
In the crystallographic case,  variables of folded cluster algebra are obtained by specialization of the  variables of its unfolding. In particular, the same projection takes $\cc$-vectors and $\gg$-vectors of the unfolding to the  $\cc$-vectors and $\gg$-vectors of the folded cluster algebra. The latter holds for non-crystallographic quivers of finite type as well (up to a multiplication by the corresponding weight, see~\cite{DT1,DT2}).\\

  \noindent
  {\bf Notation}
  \begin{itemize}
    \item
      Given a folding of root systems $\Delta\to\Delta'$, we denote by $\Q$ an acyclic quiver (or diagram) of type $\Delta'$ (for non-crystallographic root system one arrow of $\Q$ has a non-integer weight).  Let $Q$ be the quiver of type $\Delta$ that can be folded to $\Q$, let $\mathcal A$ be the corresponding skew-symmetric cluster algebra with principal coefficients. We denote by $n$ and $n'$ the number of vertices of $Q$ and $Q'$ respectively.

    \item
     Quiver $Q$ as above defines uniquely a Coxeter element of $\Delta$. We will denote by $\A_Q$ the corresponding generalized associahedron. 

\item
  
We denote by $\p$ the folding map of $\Delta$ to $\bigcup w_k\Delta'$, where $w_k$ is the set of weights (see Table~\ref{weights}). Given $\alpha\in\Delta_{\ge -1}$, we denote by $x_\alpha$ the corresponding cluster variable of $\mathcal A$ (see~\cite{FZ2}). If $\p(\alpha)\in w\Delta'$, then $w$ is called the {\em weight} of $x_\alpha$ (or, equivalently, the weight of the indeterminant $t_\alpha$).

\item

  We denote by $P$ the map of vertices of $Q$ onto the vertices of $\Q$ induced by the folding. By default, we will denote by $[i]$ the indices of vertices of $\Q$. Given a vertex with index $[i]$ of $Q'$, the collection of indices  $i_1,\dots i_l$ of vertices of $Q$ such that $P(i_j)=[i]$ forms a block of the unfolding (see~\cite{FST2,DT1}). Thus, the indices of vertices of $\Q$ can be understood as the blocks of vertices of $Q$.


\item


 Let $P$ be the map of vertices of $Q$ onto the vertices of $\Q$ as above, assume that $Q$ belongs to the initial seed of $\mathcal A$.   We call a seed of $\mathcal A$ {\it symmetric} if it is obtained from the initial seed by iterative composite mutations compatible with $P$ (i.e., every mutation in $i_l$ comes together with  mutations in all vertices from the same block). In terms of the folding map $\p$, this is equivalent to the following: together with every variable $x_\alpha$ the seed contains all variables $x_{\tilde\alpha}$ with $\p(\tilde\alpha)=\tilde w\p(\alpha)$, $\tilde w>0$.

\item

We use the notation $x_{ki}$ for cluster variables of $\mathcal A$ (as well as for the corresponding indeterminates $t_{ki}$ and constants $c_{ki}$) as follows. The second index $i$ corresponds to the row of the Auslander-Reiten quiver (or, equivalently, to the vertex of $Q$), so $1\le i\le n$, where $n$ is the number of vertices of $Q$. The first index $k$ corresponds to the number of applications of the (inverse of the) Auslander-Reiten translation $\tau$ to the initial variables, see Example~\ref{cat-a3}. If $h$ the Coxeter number of $\Delta$, $k$ varies from $1$ to $(h+2)/2$ if $h$ is even, or from $1$ to  $(h+3)/2$ if $h$ is odd.  

In these  terms, a seed is symmetric if with every variable $x_{ki}$ the seed contains all variables $x_{kj}$ with $i,j$ belonging to the same block.



\end{itemize}

\begin{remark}
  \label{rows}
  As it was shown in~\cite{DT1,DT2}, for a given $j$ all variables $x_{kj}$ have the same weight $w_j$.

  \end{remark}

\section{Construction and main result}
\label{thm}

Consider a weighted acyclic quiver $Q'$ of type $\Delta'$ and its  weighted unfolding $Q$ of type $\Delta$ with weights $w_j$. Let $\A_Q$ be the generalized associahedron of $Q$ constructed in~\cite{BCDMTY} with parameters $c_{kj}$ satisfying $w_ic_{kj}=w_jc_{ki}$ for every pair $(i,j)$ belonging to one block.   

\begin{definition}
  \label{def_Pi}
 Let $\Pi$ be the plane given by the intersection of all hyperplanes
of the form $w_jt_{ki}=w_it_{kj}$ where $i$ and $j$ belong to the same block.

\end{definition}

\begin{remark}
  A straightforward computation shows that the requirement on parameters $c_{kj}$ above guarantees that the dimension of $\Pi$ is equal to $n'$, the rank of $Q'$. 
    \end{remark}



    \begin{example}
      \label{ex-c2}
In the notation of Example~\ref{cat-a3}, set $c_{11} =c_{13}$ and $c_{21} =c_{23} $, and assume $t_{11}=t_{13}$. Then equations in the same example imply that   $t_{21}=t_{23}$ and  $t_{31}=t_{33}$. Thus, the plane $\Pi$ has dimension $2$ with coordinates on it given by any two indeterminants $t_{ki}$ taken from the first two distinct rows corresponding to compatible cluster variables (say, $t_{11}$ and $t_{22}$). By taking the intersection of $\Pi$ with the associahedron of type $A_3$ (see Fig.~\ref{fig:A3}), we obtain the associahedron of type $C_2$ (see Fig.~\ref{fig:C2}).

\begin{figure}
\includegraphics[scale=0.6]{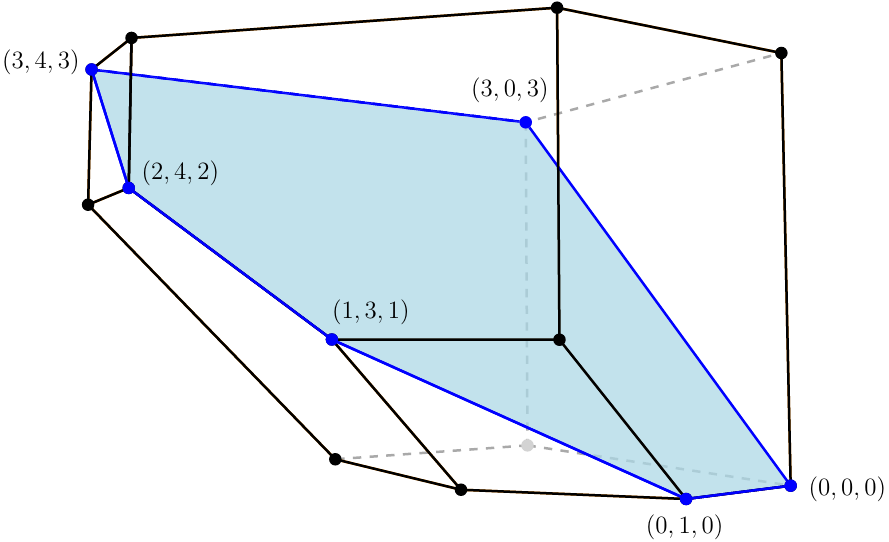}
\caption{Intersection plane $\Pi$ that yields the associahedron for $C_2$, see Example~\ref{ex-c2}.}
\label{fig:C2}
\end{figure}

\end{example}

We can now formulate our main result.

\begin{theorem} 
\label{main}  
Let  $\Delta'$ be a finite root system (possibly non-crystallographic) of rank $n'$, let $\Q$ be the a weighted acyclic quiver of type $\Delta'$.  Let $Q$ be the unfolding of $\Q$ (as listed in Table~\ref{weights}), denote by $n$ the rank of $Q$.  Consider the realization of the generalized  associahedron $\A_Q \subset \R^n$  for $Q$ constructed in~\cite{BCDMTY}. Let $\Pi\subset \R^n$ be the plane as in Definition~\ref{def_Pi}.

Then the section $\Pi\cap \A_Q$ is the generalized associahedron $\A_{Q'}$ for $\Q$.
Furthermore, the normal fan of $\A_{Q'}$ is precisely the $\gg$-vector fan of the quiver $Q'$.
\end{theorem}

To prove the theorem, we will first show that $\A_Q\cap \Pi$ is indeed a polytope of required dimension with the correct number of facets (Proposition~\ref{intersections}), and then that the $\gg$-vectors for $Q'$ are precisely normal vectors to the facets of  $\A_Q\cap \Pi$ (Propositions~\ref{normal-to-normal}--\ref{fan}). In particular, $\gg$-vectors for $Q'$ are (rescaled) orthogonal projections of $\gg$-vectors for $Q$. Proposition~\ref{fans} then shows that both $\gg$-vector fan for $Q'$ and normal fan of $\Pi\cap \A$ coincide with the intersection of $\Pi$ with the normal fan of $\A$. In view of results of~\cite{RS,YZ} this implies the theorem.


\section{Proof of the main result}
\label{proof}

We start with observing the following property of  $\Pi\cap \A_Q$.

\begin{prop}
  \label{intersections}
The plane $\Pi $ intersects every facet of $\A_Q$, and the intersection has dimension $(n'-1)$.

\end{prop}

\begin{proof}
  Every facet of $\A_Q$ is of the form $t_{ki}=0$ for some $k,i$. Assume first that either the Coxeter number $h$ of $\Delta$ is even, or $h$ is odd but $k$ is not equal to its maximal possible value, i.e. $k\ne (h+3)/2$. Then we can consider the seed containing all variables  $\{x_{kj}\mid j=1,\dots,n\}$. The coordinates on the plane are given by all values of $t_{kj}$ with $j\ne i$.  The obtained seed  is clearly symmetric, so the corresponding vertex of $\A_Q$ (which is defined by $\{t_{kj}=0\}$ for all $j=1,\dots,n$) belongs to $\Pi$. The coordinates on the intersection $\Pi\cap\A_Q$ are given by the values $\{t_{kj}\}$, where we take any one index $j$ from every block, there are precisely $n'$ blocks available. Setting  $t_{ki}=0$ we are left with $(n'-1)$-dimensional plane in $\Pi\cap\A_Q$. 

Otherwise, i.e. if $h$ is odd and $k=(h+3)/2$, we choose a symmetric seed containing $x_{ki}$ as follows. There are $n/2$ variables only with first index $k$, so we take all existing $\{x_{kj}\}$, and complement them with all $\{x_{k-1,l}\}$ such that no $x_{k,l}$ exists. We get a symmetric seed, so the further reasoning is similar to the previous case.

\end{proof}

Proposition~\ref{intersections} implies that $\Pi\cap\A_Q$ is an $n'$-dimensional polytope.

\begin{example}
  Going back to Example~\ref{ex-c2}, $t_{ki}$ in $A_3$ correspond to diagonals of a regular hexagon, seeds are given by triangulations, while symmetric seeds are precisely triangulations symmetric with respect to the center of the hexagon. Every diagonal can be included in a symmetric triangulation, the coordinate on any of the other diagonals of the triangulation parametrizes the corresponding facet of $\A_Q\cap\Pi$.


\end{example}

Denote by $\pi$ the orthogonal projection $\R^n\to\Pi$. We now want to explore the normal fan of $\A_Q\cap \Pi$. We start with the following elementary observation.

\begin{prop}
  \label{normal-to-normal}
  Let $f$ be a facet of $\A_Q\cap \Pi$, that is $f=\Pi\cap \tilde f$, where $\tilde f$ is a facet of $\A_Q$. Let $\tilde v$ be an outer  normal vector to $\tilde f$. Then $v=\pi(\tilde v)$ is an outer normal vector to $f$ in $\Pi$. 
 
\end{prop}  

\begin{proof}
  Let $v=\pi(v)+w$ where $w\in \Pi^\perp$, then for any $u\in f$ we have
$$ \langle \pi(\tilde v),u \rangle=\langle v,u\rangle -\langle w,u\rangle=0.  
$$
\end{proof}

Let $\{e_i\}$ be the basis of $\gg$-vectors for  $Q$ (i.e. the $\gg$-vectors of the initial seed given by $Q$).

The following proposition can be checked by a straightforward computation.
\begin{prop}
  \label{formula}
\begin{itemize}
\item[(a)]
  Let $i,j$ belong to one block. Then $w_j\pi(e_i)=w_i\pi(e_j)$. In particular, the vectors $\{ e_{[i]}=\pi(e_i)/w_i\}$, where we take one $i$ from every block, form a basis of $\Pi$.
\item[(b)]
  Let  $\pi: \R^n\to \Pi$  take a vector $\lambda=(\lambda_1,\dots,\lambda_n)=\sum \lambda_i e_i\in \R^n$  to the vector $ \lambda'=(\lambda_{[1]},\dots,\lambda_{[n']})=\sum \lambda_{[i]}e_{[i]}\in \Pi$. Then 
 
$$  \lambda'_{[i]}=\sum\limits_{j:\ P(j)=[i]} w_j \lambda_j.$$

\end{itemize}
  
\end{prop}

Proposition~\ref{formula} gives rise to the following notation: Given a $\gg$-vector $g_{ki}$ corresponding to variable $x_{ki}$, we denote $\pi_w(g_{ki})=\pi(g_{ki})/w_i$. 

Denote by  $\mathcal G$ (resp. $\mathcal G'$) the set of $\gg$-vectors for $Q$ (resp, $Q'$). 
Next, we prove that $\pi_w(\mathcal G)=\mathcal G'$.


\begin{prop}
\label{fan}  
The set of scaled orthogonal projections $\pi_w(\mathcal G)$ to $\Pi$ of $\gg$-vectors for $Q$ coincides with the set of $\gg$-vectors for $\Q$ (where $\gg$-vectors for $Q$ and $\Q$ are written in the bases described above). 

\end{prop}

\begin{proof}
  Indeed, in the  crystallographic case this follows immediately from  Proposition~\ref{formula} and the definition of $\gg$-vestors.
  For the non-crystallographic case combine  Proposition~\ref{formula} with Theorem~8.4 in~\cite{DT1} and Theorem~6.15 in~\cite{DT2}

\end{proof}



Proposition~\ref{fan} combined with Proposition~\ref{normal-to-normal} shows that the $\gg$-vectors for $Q'$ are precisely normal vectors to the facets of $\Pi\cap \A_Q$. Next, we will show that the normal fan of  $\Pi\cap \A_Q$ coincides with the $\gg$-vector fan for $Q'$. This is proved in the next proposition.

\begin{prop}
\label{fans}
  Both $\gg$-vector fan for $Q'$ and normal fan of $\Pi\cap \A_Q$ coincide with the intersection of $\Pi$ with the normal fan of $\A_Q$.

  \end{prop}

  \begin{proof}
    First, we prove that the intersection of $\Pi$ with the normal fan of $\A_Q$ coincides with the normal fan of $\Pi\cap \A_Q$. For this, it is sufficient to prove that given any face $f$ of $\A_Q$ intersecting $\Pi$, the intersection of its normal cone $F$ with $\Pi$  is the normal cone of the face $\Pi\cap f$ of $\Pi\cap \A_Q$. 

   Let $f$ be an $(n-q)$-dimensional face of $\A_Q$ intersecting $\Pi$, and let $F$ be the corresponding face of the normal fan of $\A_Q$. Suppose that $F$ is spanned by $\gg$-vectors  $v_{k_1j_1},\dots,v_{k_qj_q}$ of cluster variables $x_{k_1j_1},\dots,x_{k_qj_q}$ respectively. Choose any $i\in\{1,\dots,q\}$, and consider first the case when $x_{k_ij_i}$ is contained in a block of size one. Then $t_{k_ij_i}$ does not appear in the equations of $\Pi$, and thus the facet of $\A_Q$ corresponding to $t_{k_ij_i}$ is orthogonal to $\Pi$, which implies that $v_{k_ij_i}\in \Pi$, so $\pi(v_{k_ij_i})=v_{k_ij_i}\in\Pi\cap F$.     

   Now suppose that $x_{k_ij_i}$ is contained in a block with variables $x_{k_ir_2},\dots, x_{k_ir_l}$  (denote also $r_1=j_i$). Note that if not all $v_{k_ir_1},\dots, v_{k_ir_l}$ belong to $F$, then $\Pi$ does not intersect the interior of $F$, and neither it intersects the interior of $f$, so $f\cap \Pi=f'\cap \Pi$ for some proper face $f'$ of $f$. Therefore, without loss of generality we may assume that all $t_{k_ir_s}$ are contained in the list $t_{k_1j_1},\dots,t_{k_qj_q}$.    
   By Propositions~\ref{formula} and~\ref{fan}, $\pi(v_{k_ij_i})$ is a non-negative linear combination of vectors $v_{k_ir_s}$ corresponding to indeterminates $t_{k_ir_s}$, and thus  $\pi(v_{k_ij_i})\in F$.
   
Therefore, for all generating rays $v_{k_ij_i}$ of $F$  we have $\pi(v_{k_ij_i})\in F$, so $\pi(v_{k_ij_i})\in\Pi\cap F$.

Observe that $\pi(f)=(\bigcap\limits_{i=1}^q f_i)\cap\Pi=\bigcap\limits_{i=1}^q (f_i\cap\Pi)$, where $f_i$ denotes the facet of $\A_Q$ orthogonal to $v_{k_ij_i}$. By Prop.~\ref{fan}, $\pi(v_{k_ij_i})$ is a normal vector to $(f_i\cap\Pi)$, which implies that the normal cone to $\pi(f)$ is spanned by $\pi(v_{k_ij_i})$. Therefore, the normal cone to $\pi(f)$ is contained in $\Pi\cap F$. As this holds for every face of $\Pi\cap \A_Q$, this implies that the two fans coincide.
\medskip

Now, we prove that the intersection of $\Pi$ with the normal fan of $\A_Q$ coincides with the $\gg$-vector fan for $\Q$. As we have proved above, the intersection fan is spanned by $\gg$-vectors of $Q'$. 

Take any maximal cone $K$ in the intersection fan, it is an intersection of $\Pi$ with a maximal cone $\tilde K$ of the normal fan of $\A_Q$, denote the generating $\gg$-vectors of $\tilde K$ by $v_1,\dots,v_r$. Then $K$ is spanned by non-negative linear combinations of $\{v_i\}$. The only $\gg$-vectors of $Q'$ which are non-negative linear combinations of $\{v_i\}$ are projections $\pi_w(v_i)$ (see Propositions~\ref{formula},~\ref{fan}), so $K$ is a cone of the $\gg$-vector fan of $Q'$. The statement now follows since all the fans in question are complete. 

\end{proof}

\begin{remark}
It was pointed out to us by Nathan Reading that the fact that  the $\gg$-vector fan for $Q'$ coincides with the intersection of $\Pi$ with the $\gg$-vector fan for $Q$ is also proved by Viel in~\cite[Theorem 2.4.24]{V}.
  
\end{remark}

We are now ready to complete the proof of the main theorem. 

Proposition~\ref{fans} shows that the normal fan of the polytope $\Pi\cap\A_Q$ is precisely $\gg$-vector fan of $Q'$. According to the general form of~\cite[Theorem 10.2]{RS} (see also~\cite[Theorem 1.10]{YZ}), this implies that $\Pi\cap\A_Q=\A_{Q'}$, which proves  Theorem~\ref{main} in the crystallographic cases (as Theorem 10.2 in~\cite{RS} is stated for crystallographic case only).

In the non-crystallographic cases we proceed as follows. The proof of (the general form of)~\cite[Theorem 10.2]{RS} is based on the fact that mutations of $\gg$-vectors are given by Conjecture~7.12 of~\cite{FZ4}, while~\cite[Conjecture~7.12]{FZ4} is implied by the sign-coherence of $\cc$-vectors~\cite{NZ} (note that in the finite types the $\cc$-vectors are manifestly sign-coherent as they are all roots of the corresponding root system). As defined in~\cite{DT1,DT2}, mutations of $\gg$-vectors for types $H_3,H_4$ and $I_n$ also satisfy Conjecture~7.12 of~\cite{FZ4}, so all considerations above can be applied.


\begin{thebibliography}{BCDMTY}

\bibitem[ABHY18]{ABHY}  
      N.~Arkani-Hamed, Y.~Bai, S.~He, G.~Yan, 
{\em Scattering forms and the positive geometry of kinematics, color and the worldsheet}, 
JHEP (5) 96 (2018).
  
\bibitem[AHL21]{AHL}
      N.~Arkani-Hamed, S.~He, T.~Lam, 
{\em Cluster configuration spaces of finite type}, 
SIGMA 17 (2021), 41pp.

\bibitem[BCDMTY23]{BCDMTY}
      V.~Bazier-Matte, N.~Chapelier-Laget, G.~Douville, K.~Mousavand, H.~Thomas, E.~Y\i ld\i r\i m,
{\em ABHY associahedra and Newton polytopes of F-polynomials for cluster algebras of simply laced finite type}, 
J. Lond. Math. Soc.  (2023). 

\bibitem[CSZ15]{CSZ} 
      C.~Ceballos, F.~Santos, G.M.~Ziegler,
{\em Many non-equivalent realizations of the associahedron},
 Combinatorica 35 (2015), 513--551.

\bibitem[CFZ02]{CFZ}
F.~Chapoton, S.~Fomin, A.~Zelevinsky,
{\em Polytopal realizations of generalized associahedra},
Can. Math. Bull. 45  (2002), 537--566.

\bibitem[Cr99]{Cr} J. Crips, {\em Injective maps between Artin groups}, In: Geometric group theory down under (Canberra, 1996), de Gruyter, Berlin (1999) 119--137.


\bibitem[DT22]{DT1} 
      D.~D.~Duffield, P.~Tumarkin,  
{\em Categorifications of non-integer quivers: types $H_4$, $H_3$ and $I_2(2n+1)$}, 
arXiv:2204.12752.

\bibitem[DT23]{DT2}  
      D.~D.~Duffield, P.~Tumarkin, 
{\em Categorifications of non-integer quivers: type $I_2(2n)$}, 
arXiv:2302.06988.


\bibitem[D08]{Du}
      G.~Dupont, 
{\em An approach to non-simply laced cluster algebras},  
J. Algebra 320 (2008), 1626--1661.

\bibitem[FST12]{FST2}
      A.~Felikson, M.~Shapiro, P.~Tumarkin,
{\em Cluster algebras of finite mutation type via unfoldings}, 
Int. Math. Res. Notices (2012), 1768--1804. 

  
\bibitem[FR07]{FR} 
      S.~Fomin, N.~Reading, 
{\em Root systems and generalized associahedra}, 
Geometric Combinatorics (13) 2007 (IAS/Park City Math. Ser.), 63--131.

\bibitem[FZ03a]{FZY} 
      F.~Fomin, A.~Zelevinsky, 
{\em Y-systems and generalized associahedra}, 
Ann. Math. 158 (2003), 977--1018.

\bibitem[FZ03b]{FZ2} 
      F.~Fomin, A.~Zelevinsky, 
{\em Cluster algebras {II}: {F}inite type classification}, 
Invent. Math. 154 (2003), 63--121.

\bibitem[FZ07]{FZ4} 
      F.~Fomin, A.~Zelevinsky, 
{\em Cluster algebras. {IV}. {C}oefficients}, 
Compos. Math. 143 (2007), 112--164.



\bibitem[HLT11]{HLT} C.~Hohlweg, C.~Lange, H.~Thomas, {\em Permutahedra and generalized associahedra},
Adv. Math. 226 (2011), 608--640.


\bibitem[HPS18]{HPS} C.~Hohlweg, V.~Pilaud, S.~Stella, {\em Polytopal realizations of finite type $\gg$-
vector fans}, Adv. Math. 328 (2018), 713--749.


\bibitem[L83]{L} G. Lusztig, {\em Some examples of square integrable representations of semisimple $p$-adic groups}, Trans. Amer. Math. Soc., 277 (1983), 623--653.


\bibitem[MRZ03]{MRZ} B.~R.~Marsh, M.~Reineke, A.~Zelevinsky, {\em Generalized associahedra via quiver representations}, Trans. Amer. Math. Soc. 355 (2003), 4171--4186.

  
  
\bibitem[MP93]{MP} 
      R.~V.~Moody, J.~Patera, 
{\em Quasicrystals and icosians}, 
J. Phys. A 26 (1993), 2829--2853.

\bibitem[NZ12]{NZ} T.~Nakanishi, A.~Zelevinsky, {\em On tropical dualities in cluster algebras}, Contemp. Math. 565 (2012), 217--226.

  
\bibitem[PPPP23]{P4} 
      A.~Padrol, Y.~Palu, V.~Pilaud, P.-G.~Plamondon, 
{\em Associahedra for finite-type cluster algebras and minimal relations between $\gg$-vectors}, 
Proc. London Math. Soc. 127 (2023), 513--588.

\bibitem[PSZ23]{PSZ} 
      V.~Pilaud, F.~Santos, G.~M.~Ziegler, 
{\em Celebrating Loday's associahedron}, 
Arch. Math. 121 (2023), 559--601.
  
\bibitem[R06]{R} 
      N.~Reading, 
{\em Cambrian lattices}, 
 Adv. Math. 205 (2006), 313--353.


\bibitem[R07]{R2} 
      N.~Reading, {\em Clusters, Coxeter-sortable elements and noncrossing partitions}, 
Trans. Amer. Math. Soc. 359 (2007), 5931--5958.

 
\bibitem[RS09]{RS} 
      N.~Reading, D.~Speyer, 
{\em Cambrian fans}, 
J. Eur. Math. Soc. 11 (2009), 407--447.
  
  
\bibitem[QZ23]{QZ} 
      Y.~Qiu, X.~Zhang, 
{\em Fusion-stable structures on triangulation categories}, 
arXiv:2310.02917.

\bibitem[Sta63]{St} 
J.~D.~Stasheff,
{\em Homotopy associativity of H-spaces I \& II},
Trans. Amer. Math. Soc. 108 (1963), 275--312.
  
\bibitem[Stel13]{S} 
      S.~Stella, 
{\em Polyhedral models for generalized associahedra via Coxeter elements},
J. Algebraic Combin. 38 (2013), 121--158.

\bibitem[Stem08]{Stem}
  J.~R.~Stembridge, {\em Folding by automorphisms}, 2008.\\ URL: \href{https://dept.math.lsa.umich.edu/~jrs/papers/folding.pdf}{https://dept.math.lsa.umich.edu/~jrs/papers/folding.pdf}


\bibitem[T51]{T} 
      D.~Tamari,
{\em Monoides pr\'eordonn\'es et cha\^ines de Malcev}, PhD thesis, 
Université Paris Sorbonne, 1951.

\bibitem[V18]{V} S.~Viel, {\em Cluster algebras and mutation-linear algebra: folding,
    dominance, and the orbifolds model}, Ph.D. thesis, NCSU, 2018.


\bibitem[YZ08]{YZ} S.~W.~Yang, A.~Zelevinsky, {\em Cluster algebras of finite type via Coxeter elements and principal minors}, Transform. Groups 13 (2008), 855--895. 

  
\end{thebibliography}
\end{document}